\documentclass{amsart}

\usepackage{amsmath, amsthm, amscd, amsfonts, amssymb, enumerate, verbatim, newlfont, calc, graphicx, color}
\usepackage[bookmarksnumbered, colorlinks, plainpages]{hyperref}
\usepackage{tikz} 

\newtheorem{theorem}{Theorem}[section]
\newtheorem{question}[theorem]{Question}
\newtheorem{lemma}[theorem]{Lemma}
\newtheorem{proposition}[theorem]{Proposition}
\newtheorem{corollary}[theorem]{Corollary}
\theoremstyle{definition}
\newtheorem{definition}[theorem]{Definition}
\newtheorem{fact}[theorem]{Fact}

\newtheorem{example}[theorem]{Example}

\usepackage{pstricks}
\usepackage{epsfig}
\usepackage{pst-grad}
\usepackage{pst-plot}

\begin{document}
\author[M.  Sayedsadeghi,  M. Nasernejad, and A. A. Qureshi]{Mirsadegh ~ Sayedsadeghi$^{1}$,  Mehrdad ~Nasernejad$^{2,*}$, and Ayesha Asloob Qureshi$^{3}$}
\title[On the embedded associated primes of monomial ideals]{On the embedded associated primes of monomial ideals}
\thanks{$^*$Corresponding author}

\thanks{Mirsadegh  Sayedsadeghi  was supported by the grant of Payame Noor University of Iran.}

\thanks{Mehrdad Nasernejad  was in part supported by a grant from IPM (No. 991300120).}

\thanks{E-mail addresses:  msayedsadeghi@gmail.com, m\_{nasernejad@yahoo.com}, and aqureshi@sabanciuniv.edu}  
\maketitle
 
\begin{center}
{\it
$^1$Department of Mathematics, Faculty of Science, Payame Noor \\ University (PNU), P.O. Box, 19395-3697, Tehran, Iran\\

$^2$School of Mathematics, Institute for Research  in  Fundamental \\
Sciences (IPM),  P.O. Box 19395-5746, Tehran, Iran\\
 $^3$Sabanc\i\;University, Faculty of Engineering and Natural Sciences, \\
Orta Mahalle, Tuzla 34956, Istanbul, Turkey
}
\end{center}

\vspace{0.4cm}

\begin{abstract}
 Let $I$  be a square-free  monomial ideal in a  polynomial ring   $R=K[x_1,\ldots, x_n]$  over a field $K$, $\mathfrak{m}=(x_1, \ldots, x_n)$ be  the  graded maximal ideal of $R$,   and  $\{u_1, \ldots, u_{\beta_1(I)}\}$  be a maximal independent set of minimal generators of $I$ such that  $\mathfrak{m}\setminus x_i \notin \mathrm{Ass}(R/(I\setminus x_i)^t)$ for all   $x_i\mid \prod_{i=1}^{\beta_1(I)}u_i$ and some positive integer $t$, where $I\setminus x_i$ denotes the  deletion of $I$ at $x_i$ and $\beta_1(I)$  denotes the maximum cardinality of an independent set in $I$.
In this paper, we prove that if  $\mathfrak{m}\in \mathrm{Ass}(R/I^t)$, then $t\geq \beta_1(I)+1$.  As an application, we verify that under certain conditions, every unmixed K$\mathrm{\ddot{o}}$nig ideal is normally torsion-free, and so has the strong persistence property. 
In addition, we show that every square-free transversal polymatroidal ideal is normally torsion-free. 
Next, we state some results on the corner-elements of monomial ideals. In particular, we prove that if  $I$ is  a monomial ideal in a polynomial ring $R=K[x_1, \ldots, x_n]$ over a field $K$ and $z$ is  an $I^t$-corner-element  for some positive integer $t$  such that  $\mathfrak{m}\setminus x_i \notin \mathrm{Ass}(I\setminus x_i)^t$ for some $1\leq i \leq n$, then  $x_i$ divides $z$.
\end{abstract}
\vspace{0.4cm}

\section{Introduction}
Suppose that   $R$ is    a commutative Noetherian ring,  $I$ is  an ideal of $R$, and $\mathrm{Ass}_R(R/I)$ denotes  the set of all prime ideals associated to  $I$.
  Brodmann \cite{BR} proved  that the sequence $\{\mathrm{Ass}_R(R/I^k)\}_{k \geq 1}$ of associated prime ideals is stationary  for large $k$, i.e., there exists a positive integer $k_0$ such that $\mathrm{Ass}_R(R/I^k)=\mathrm{Ass}_R(R/I^{k_0})$ for all $k\geq k_0$. The  minimum  such $k_0$ is called the
{\it index of stability  of} $I$ and $\mathrm{Ass}_R(R/I^{k_0})$ is called the {\it stable set of associated prime ideals of} $I$, which is denoted by $\mathrm{Ass}^{\infty }(I)$, see \cite{KHN} for more details.      

Many   questions   arise in the context of Brodmann's result.  
Recall that if  $I$ is  an ideal in a commutative Noetherian ring $R$, then  $I$ is said to have the {\it persistence property} if $\mathrm{Ass}(R/I^k)\subseteq \mathrm{Ass}(R/I^{k+1})$ 
for all positive integers $k$. Moreover, an ideal $I$ satisfies the {\it strong persistence property} if $(I^{k+1}: I)=I^k$ for all positive integers $k$,  for more details refer to \cite{ N2, NKA}.  Furthermore,  we say that $I$  has the {\it symbolic strong persistence property}  if 
$(I^{(k+1)}: I^{(1)})=I^{(k)}$ for all $k$, where $I^{(k)}$ denotes the  $k$-th symbolic  power  of $I$, cf. \cite{RT}. 
An ideal $I$ is called {\it normally torsion-free} if $\mathrm{Ass}(R/I^k) \subseteq \mathrm{Ass}(R/I)$ for all $k$, see \cite[Definition 1.4.5]{HH1}.  In particular,   Kaiser, Stehl$\mathrm{\acute{i}}$k, 
and $\mathrm{\check{S}}$krekovski \cite{KSS} have shown that not all square-free monomial ideals have the persistence property. However,  by applying combinatorial methods, it has  been shown  that many large families of square-free monomial ideals  satisfy the persistence property and the strong persistence property. It has been shown that the persistence property and also the strong persistence property holds  for edge ideals of finite  simple graphs \cite {MMV}, edge ideals of  finite   graphs with loops \cite{RT},  and  polymatroidal ideals \cite{HRV}. Furthermore,  according to \cite{FHV2},   cover ideals of perfect graphs have the persistence property. In addition, a  few  examples of normally torsion-free monomial ideals appear from graph theory, see \cite{ N3, KHN1}.  In \cite{SVV}, it has been  already proved that  a finite simple graph $G$ is bipartite if and only if its edge ideal is normally torsion-free. Moreover, by \cite{GRV}, it is well-known that the cover ideals of bipartite graphs are normally torsion-free. In addition, in \cite{HRV},  it has been verified  that every  transversal polymatroidal ideal  is normally torsion-free.  One of our motivations in this paper is to give a large class of square-free monomial ideals which satisfy  normality, normally torsion-freeness, and  
the (symbolic)  (strong)  persistence property. 
\par 
Now, let $I$ be a monomial ideal in a polynomial ring $R=K[x_1, \ldots, x_n]$, and  $\mathfrak{m}=(x_1, \ldots, x_n)$ be the  graded maximal ideal of $R$. 
One motivating  question in this field is the existence of the graded maximal ideal in the set of associated primes. It should be noted that  little is known on this subject in literature,  see 
\cite[Lemma 3.1]{CMS}, \cite[Corollary 3.6]{HM}, and  
\cite[Theorem 2.7]{NR} for more details. As another motivation of  this paper,  we give  some results in this theme. 

This paper is organized as follows: In Section 2 we establish notation and  definitions which  appear  throughout the paper.
In Section 3,  we concentrate on the embedded associated primes of monomial ideals. More precisely, 
let $I$  be a square-free  monomial ideal in a  polynomial ring   $R=K[x_1,\ldots, x_n]$  over a field $K$, $\mathfrak{m}=(x_1, \ldots, x_n)$ be  the  graded maximal ideal of $R$,   and 
 $\{u_1, \ldots, u_{\beta_1(I)}\}$  be a maximal independent set of minimal generators of $I$ such that  $\mathfrak{m}\setminus x_i \notin \mathrm{Ass}(R/(I\setminus x_i)^t)$ for all   $x_i\mid \prod_{i=1}^{\beta_1(I)}u_i$ and some positive integer $t$, where $I\setminus x_i$ denotes the  deletion of $I$ at $x_i$ and $\beta_1(I)$  denotes the maximum cardinality of an independent set in $I$ (see  \cite[page 303]{HM} for  definitions of deletion and independent set).
 If $\mathfrak{m}\in \mathrm{Ass}(R/I^t)$, then $t\geq \beta_1(I)+1$, cf. Corollary \ref{Main.Corollary}. As an application of Corollary \ref{Main.Corollary}, we verify that under certain conditions,  every unmixed K$\mathrm{\ddot{o}}$nig ideal is normally torsion-free, and so has the strong persistence property, see Theorem \ref{Application}. In addition,   Theorem \ref{Th.Main2} implies  that under certain conditions a square-free monomial ideal is normally torsion-free. 
We conclude  by establishing the claim that 
every square-free transversal polymatroidal ideal is normally torsion-free, see Theorem \ref{Application2}.

Section 4 is devoted to the investigation of corner-elements of monomial ideals. In particular,  in  Corollary \ref{Main.Corollary2}, it is proved  that if 
 $I$ is  a monomial ideal in a polynomial ring $R=K[x_1, \ldots, x_n]$ over a field $K$ and $z$ is  an $I^t$-corner-element  for some positive integer $t$  such that  $\mathfrak{m}\setminus x_i \notin \mathrm{Ass}(I\setminus x_i)^t$ for some $1\leq i \leq n$, then  $x_i$ divides $z$.

\section{Preliminaries}

In this section, we introduce notation  and definitions which will be used in the rest of this paper. \par 

Assume that $R$ is  a commutative Noetherian ring and $I$ is  an ideal of $R$. A prime ideal $\mathfrak{p}\subset  R$ is an {\it associated prime} of $I$ if there exists an element $h$ in $R$ such that $\mathfrak{p}=(I:_R h)$, where $(I:_Rh)=\{r\in R |~  rh\in I\}$. The  {\it set of associated primes} of $I$, denoted by  $\mathrm{Ass}_R(R/I)$, is the set of all prime ideals associated to  $I$. It is known  that if $I$ is an ideal in  a commutative Noetherian ring $R$, then  $\mathrm{Ass}_R(R/I)$ is a finite set, cf.  \cite[Theorem 6.5]{Mat}.\par

A presentation of an ideal $I$ is an intersection $I=\cap_{i=1}^k\mathfrak{q}_i$, where  each $\mathfrak{q}_i$ is a primary ideal,  is called a {\it primary decomposition} of $I$. Let $\mathrm{Ass}_R(R/\mathfrak{q}_i)=\mathfrak{p}_i$ for each $i=1, \ldots, k$. The primary decomposition is called {\it minimal  primary decomposition} if none of the $\mathfrak{q}_i$ can be omitted in this intersection, and also $\mathfrak{p}_i \neq \mathfrak{p}_j$ for all $i\neq j$.
If $I=\cap_{i=1}^k\mathfrak{q}_i$ is a minimal  primary decomposition of $I$, then  $\mathrm{Ass}_R(R/I)=\{\mathfrak{p}_1, \ldots, \mathfrak{p}_k\}$. The minimal members of 
$\mathrm{Ass}_R(R/I)$ are called the {\it minimal} primes of $I$, and $\mathrm{Min}(I)$ denotes the set of minimal prime ideals of $I$. Also, the associated primes of $I$ which are not minimal are called the {\it embedded} primes of $I$. 
In particular,  if $I$ is a square-free monomial ideal, then $\mathrm{Ass}_R(R/I)=\mathrm{Min}(I)$,  refer to 
 \cite[Corollary 1.3.6]{HH1}. \par

In addition, the {\it $k$-th symbolic power} of $I$, denoted by $I^{(k)}$, is the intersection of those primary components of $I^k$ which belong to the minimal prime ideals of $I$, cf. \cite[page 14]{HH1}. \par 

Let $I$ be a square-free monomial ideal and $\Gamma \subseteq \mathcal{G}(I)$, where $\mathcal{G}(I)$ denotes the unique minimal set of monomial generators of the  monomial ideal $I$. We say that $\Gamma$ is an {\it independent} set in $I$ if $\mathrm{gcd}(f,g)=1$ for each $f,g\in \Gamma$ with $f\neq g$. We denote the maximum cardinality of an independent set in $I$ by $\beta_1(I)$. 
Furthermore, if $I$ is a monomial ideal, then the  {\it deletion} of $I$ at $x_i$ with $1\leq i \leq n$, denoted by $I\setminus x_i$, is obtained by setting $x_i=0$ in every minimal generator of $I$, that is, we delete every minimal generator such as $u\in \mathcal{G}(I)$ with $x_i\mid u$, see  \cite[page 303]{HM} for  definitions of deletion and independent set.\par

Recall from \cite[Definition 6.1.5]{V1} that if  $u=x_{1}^{a_{1}}\cdots x_{n}^{a_{n}}$ is  a
monomial in a polynomial ring $R=K[x_{1},\ldots ,x_{n}]$ over a field $K$, then 
the \textit{support} of $u$ is given by $\mathrm{supp}(u):=\{x_{i}|~a_{i}>0%
\} $.\par 

Notice  that a {\it  K$\ddot{o}$nig} ideal is a square-free monomial ideal $I$ with  $\mathcal{G}(I)=\{u_1, \ldots, u_r\}$ such that the maximum number of pairwise disjoint monomials $u_1, \ldots, u_r$ is equal to the height of $I$. Note that edge ideals of bipartite graphs form a large class of K$\mathrm{\ddot{o}}$nig ideals. \par 

Let $R$ be a unitary commutative ring and $I$ an ideal in $R$. An element $f\in R$ is {\it integral} over $I$, if there exists an equation 
 $$f^k+c_1f^{k-1}+\cdots +c_{k-1}f+c_k=0 ~~\mathrm{with} ~~ c_i\in I^i.$$
 The set of elements $\overline{I}$ in $R$ which are integral over $I$ is the 
 {\it integral closure} of $I$. The ideal $I$ is {\it integrally closed}, if $I=\overline{I}$, and $I$ is {\it normal} if all powers of $I$ are integrally closed, refer to \cite{ANR} for more information. \par 

The following   facts  will be used in the rest of this paper:
\begin{fact} \label{fact1}\cite[Exercise 6.1.23]{V1} 
If $I$, $J$, $L$ are monomial ideals, then  the following equalities hold:
\begin{itemize}
\item[(i)]  $I\cap (J+L)=(I \cap J) + (I \cap L).$
\item[(ii)] $I+ (J \cap L)= (I+J) \cap (I+L).$
\end{itemize}
\end{fact}

\begin{fact}\label{fact2} \cite[Exercise 2.1.62]{V1}
Let $R$  be a ring and $I$ an ideal. If $x \in R \setminus I$, 
 then there is an exact sequence of $R$-modules: 
 $$0\longrightarrow R/(I:x)\stackrel{\psi}{\longrightarrow} R/I \stackrel{\phi}{\longrightarrow} R/(I,x) \longrightarrow  0,$$
 where $\psi(\overline{r})= x\overline{r}$ is multiplication by $x$ and 
 $\phi(\overline{r}) =\overline{r}$.
\end{fact}

\begin{fact}\label{fact3} \cite[Exercise 9.42]{sharp}
Let 
$$0\longrightarrow L{\longrightarrow} M {\longrightarrow} N  \longrightarrow  0,$$

be a short exact sequence of modules and homomorphisms over the commutative Noetherian ring $R$. Then, 
$\mathrm{Ass}(L) \subseteq  \mathrm{Ass}(M) \subseteq 
 \mathrm{Ass}(L) \cup  \mathrm{Ass}(N).$ 
\end{fact}
\begin{fact}\label{fact4} \cite[Lemma  9.38]{sharp}
Let $M$  be a module over the commutative Noetherian ring
$R$, and let $S$ be a multiplicatively closed subset of $R$. Then
$$\mathrm{Ass}_{S^{-1}R}(S^{-1}M)=\{\mathfrak{p}S^{-1}R: \mathfrak{p}\in \mathrm{Ass}_R(M) ~~ \mathrm{and} ~~ \mathfrak{p} \cap S =\emptyset \}.$$
\end{fact}
\begin{fact} \label{fact5}\cite[Lemma 3.12]{SN} 
 Let  $I$ be a monomial ideal in a polynomial ring $R=K[x_1, \ldots, x_n]$ with 
 $\mathcal{G}(I)=\{u_1, \ldots, u_m\}$, and $h=x_{j_1}^{b_1}\cdots x_{j_s}^{b_s}$ with $j_1, \ldots, j_s \in \{1, \ldots, n\}$ be a monomial in $R$. Then  $I$ is normally torsion-free  if and only if $hI$ is normally torsion-free.
\end{fact}

Throughout this paper,  we denote the unique minimal set of monomial generators of a  monomial ideal $I$  by $\mathcal{G}(I)$. Also, $R=K[x_1,\ldots, x_n]$ is a polynomial ring over a field $K$, $\mathfrak{m}=(x_1, \ldots, x_n)$ is  the  graded maximal ideal of $R$, and $x_1, \ldots, x_n$ are indeterminates. The symbol $\mathbb{N}$  will always denote the set of positive integers.

\section{Some results on the embedded associated primes of monomial ideals}
In this section, we focus on the embedded associated primes of monomial ideals. The main results of this section are Corollary  \ref{Main.Corollary} and Theorem \ref{Th.Main2}. To achieve this goal,  we begin  with the following lemma which is essential  for the proofs of  Proposition  
\ref{reduction} and Theorem \ref{disjoint}.

\begin{lemma} \cite[Lemma 2.1]{KHN2} \label{Kaplansky}
Let $S$ be a commutative ring and let $a_1,\ldots,a_m$ be elements constituting a permutable  $S$-sequence. Let $J$
 be an ideal generated by monomials in $a_{t+1},\ldots,a_m$  for some $t\in\mathbb{N}$ with $1\leq t \leq m-1$. Then $(J:_S a_1^{n_1} a_2^{n_2}\ldots a_t^{n_t})=J$ for all $n_1,n_2,\ldots,n_t\in\mathbb{N}$.
\end{lemma}
\begin{proof}
Proceed by induction on $t$. In the case in which $t=1$ follows from \cite[Lemma 3]{KA}. Now, suppose, inductively, that $t>1$ and that the result has been proved for $t$. Then, by using Lemma 3 in \cite{KA} in conjunction with induction hypothesis, one has  the following equalities 
$$(J:_S a_1^{n_1} a_2^{n_2}\ldots a_{t+1}^{n_{t+1}})=((J:_S a_1^{n_1} a_2^{n_2}\ldots a_t^{n_t}):_S a_{t+1}^{n_{t+1}})=(J:_S a_{t+1}^{n_{t+1}})=J.$$
\end{proof}

 

To prove   Theorem \ref{disjoint}, one must  apply the subsequent proposition. 

\begin{proposition}\label{reduction}
Let  $I$ be   a monomial ideal in a  polynomial ring $R=K[x_1, \ldots, x_n]$ over a field $K$ and $\mathfrak{p}\in \mathrm{Ass}(R/I)$. Then there exists a  monomial $v$ in  $R$ with $\mathrm{supp}(v)\subseteq \cup_{u\in \mathcal{G}(I)} \mathrm{supp}(u)$ such that $\mathfrak{p}=(I:_R v)$.
\end{proposition}
\begin{proof}
By virtue of \cite[Corollary 1.3.10]{HH1}, there exists a monomial $w$ in  $R$ such that $\mathfrak{p}=(I:_R w)$. 
Without loss of generality, one may assume that 
$\cup_{u\in \mathcal{G}(I)} \mathrm{supp}(u)=\{x_1, \ldots, x_m\}$. Certainly, one can write $w=w_1w_2$, where  
$\mathrm{supp}(w_1)\subseteq \{x_1, \ldots, x_m\}$ and 
$\mathrm{supp}(w_2)\subseteq \{x_{m+1}, \ldots, x_n\}$. 
It follows from Lemma \ref{Kaplansky} that $(I:_Rw_2)=I$, and so we get the following equalities 
$$\mathfrak{p}=(I:_Rw)=((I:_Rw_2):_Rw_1)=(I:_Rw_1).$$ 
This completes our argument.
\end{proof}


The next theorem enables us to argue on the associated prime ideals of a monomial ideal which can be divided into two monomial ideals generated in disjoint sets of variables. 

\begin{theorem}\label{disjoint}
 Let $I_1\subset R_1=K[x_1, \ldots, x_n]$ and  
$I_2\subset R_2=K[y_1, \ldots, y_m]$ be two monomial ideals in disjoint sets of variables. Let 
$$I=I_1R+I_2R\subset R=K[x_1, \ldots, x_n, y_1, \ldots, y_m].$$ 
Then $\mathfrak{p}\in \mathrm{Ass}(R/I)$ if and only if 
$\mathfrak{p}=\mathfrak{p}_1R + \mathfrak{p}_2R$, where 
$ \mathfrak{p}_1\in \mathrm{Ass}(R_1/I_1)$ 
and $ \mathfrak{p}_2\in \mathrm{Ass}(R_2/I_2)$.
\end{theorem}

\begin{proof}
Let $I_1=Q_1\cap \cdots \cap Q_s$ be a minimal primary decomposition of $I_1$ such that $\sqrt{Q_i}=\mathfrak{p}_i$ for all $i=1,\ldots,s$. In view of  Fact~\ref{fact1}, we  have 
 $$I=\bigcap_{i=1}^s Q_i+I_2=\bigcap_{i=1}^s (Q_i+I_2).$$ 
 In addition, it is routine to check that  there exists the following  $R$-monomorphism
 $$\theta : R/I \rightarrow \bigoplus_{i=1}^s R/(Q_i+I_2),$$ given by $\theta(r+I)=(r+Q_1+I_2,\ldots, r+Q_s+I_2)$ for all $r\in R$.
This implies that    $$\mathrm{Ass}(R/I) \subseteq \bigcup_{i=1}^s \mathrm{Ass}(R/(Q_i+I_2)).$$ 
Assume that  $I_2=Q'_1\cap \cdots \cap Q'_t$ is  a minimal primary decomposition of $I_2$ such that $\sqrt{Q'_j}=\mathfrak{p}'_j$ for all $j=1,\ldots,t$.
Now,  fix $1\leq i \leq s$. Applying  Fact~\ref{fact1} yields 
  $$Q_i+I_2= Q_i+ \bigcap_{j=1}^t Q'_j=\bigcap_{j=1}^t
(Q_i+Q'_j).$$ 
Consequently,  there exists the 
 following $R$-monomorphism 
 $$\phi : R/(Q_i+I_2) \rightarrow \bigoplus_{j=1}^t R/(Q_i+Q'_j),$$
given by $\phi(r+Q_i+I_2)=
(r+ Q_i +Q'_1, \ldots, r+Q_i+Q'_t)$ for  all  $r\in R$. This implies   $$\mathrm{Ass}(R/(Q_i+ I_2)) \subseteq \bigcup_{j=1}^t \mathrm{Ass}(R/(Q_i + Q'_j)).$$
  Since $Q_i+Q'_j$ is a $(\mathfrak{p}_i+\mathfrak{p}'_j)$-primary monomial ideal, for all  $j=1,\ldots, t$, one has 
$\mathrm{Ass}(R/(Q_i+ Q'_j))= \mathfrak{p}_i+\mathfrak{p}'_j $  for all  $i=1,\ldots,s$ and $j=1,\ldots, t$.  Therefore,  
  $$\mathrm{Ass}(R/I)\subseteq\{\mathfrak{p} \ |\ \mathfrak{p} = \mathfrak{p}_i +  \mathfrak{p}'_j \ \mathrm{for} \ \mathrm{all} \ i=1,\ldots, s \ \mathrm{and} \ j=1,\ldots, t\}.$$
 To establish the  reverse inclusion, pick arbitrary  elements  $\mathfrak{p} \in \mathrm{Ass}(R/I_1)$ and $\mathfrak{p}' \in \mathrm{Ass}(R/I_2)$. Then there exist  monomials
 $v$ and $v'$ in  $R$ such that $\mathfrak{p}=(I_1:_R v)$ and   $\mathfrak{p}'=(I_2:_R v')$. Based on  Proposition \ref{reduction},  we may assume that 
 \begin{equation}
 \displaystyle\mathrm{supp}(v)\subseteq \bigcup_{f\in \mathcal{G}(I_1)}\mathrm{supp}(f) \text{~~and~~} \mathrm{supp}(v')\subseteq \bigcup_{g\in \mathcal{G}(I_2)}\mathrm{supp}(g).\label{11}
 \end{equation}
 In  light of (\ref{11}) and Lemma \ref{Kaplansky}, one can deduce that $(I_1:_Rvv')=(I_1:_Rv)$ and  $(I_2:_Rvv')=(I_2:_Rv')$. Accordingly, we obtain the following equalities
 
   \begin{align*}
  \mathfrak{p} + \mathfrak{p}' & = (I_1:_R v) + (I_2:_R v')\\
   &= (I_1:_R vv') + (I_2:_R vv')   \\
 &=(I_1+I_2:_Rvv') \\
  &= (I:_Rvv').
 \end{align*} 
 This implies that   $\mathfrak{p} + \mathfrak{p}' \in \mathrm{Ass}_R(R/I)$. Thus, the reverse inclusion holds. 
\end{proof}


To prove Corollary \ref{Main.Corollary}, we need to use the following  crucial theorem. Moreover, it should be noted that  although  checking the hypotheses of the following theorem is not easy practically, but  using Macaulay 2 software \cite{GS} may be helpful. 

\begin{theorem} \label{Th.Main}
Let $I\subset R=K[x_1, \ldots, x_n]$ be a  monomial ideal, $\mathfrak{m}=(x_1, \ldots, x_n)$, $t$ a positive integer,  and $y_1, \ldots, y_s$ be distinct variables in $R$  such that, 
for each  $i=1, \ldots, s$, $\mathfrak{m}\setminus y_i \notin \mathrm{Ass}(R/(I\setminus y_i)^t)$, where $I\setminus y_i$ denotes the  deletion of $I$ at $y_i$. 
 Then  $\mathfrak{m}\in \mathrm{Ass}(R/I^t)$ if and only if $\mathfrak{m}\in \mathrm{Ass}(R/(I^t:\prod_{i=1}^sy_i))$. 
\end{theorem}
\begin{proof}
We first assume that 
$\mathfrak{m}\in \mathrm{Ass}(R/(I^t:\prod_{i=1}^sy_i))$. 
Based on Fact~\ref{fact2}, we get  the following short exact sequence
$$0\longrightarrow R/(I^t:\prod_{i=1}^sy_i)\stackrel{\psi}{\longrightarrow} R/I^t \stackrel{\phi}{\longrightarrow} R/(I^t, \prod_{i=1}^sy_i) \longrightarrow  0,$$
where $\psi(\overline{r})=\overline{r}\prod_{i=1}^sy_i$ and $\phi(\overline{r})=\overline{r}$. In addition, by virtue of Fact~\ref{fact3}, one can derive that 
\begin{equation}
\mathrm{Ass}(R/(I^t:\prod_{i=1}^sy_i))\subseteq \mathrm{Ass}(R/I^t) \subseteq \mathrm{Ass}(R/(I^t:\prod_{i=1}^sy_i)) \cup \mathrm{Ass}(R/(I^t, \prod_{i=1}^sy_i)). \label{12}
\end{equation}
Since  $\mathfrak{m}\in \mathrm{Ass}(R/(I^t:\prod_{i=1}^sy_i))$, we get  
$\mathfrak{m}\in \mathrm{Ass}(R/I^t)$. 
To establish the converse implication, let $\mathfrak{m}\in \mathrm{Ass}(R/I^t)$.  Our aim is to prove that  
$\mathfrak{m}\in \mathrm{Ass}(R/(I^t:\prod_{i=1}^sy_i))$. 
Proceed  by  induction on $s$.  
Let $J:=I\setminus y_1$. Because the generators of $J^t$ are precisely the generators of $I^t$ that are not divisible by $y_1$, this implies that $(I^t,y_1)=(J^t,y_1)$. By  assumption, one has $\mathfrak{m}\setminus y_1 \notin \mathrm{Ass}(R/J^t)$. Also, it follows from Theorem \ref{disjoint}  that  $\mathfrak{p}\in \mathrm{Ass}(R/(J^t, y_1))$ if and only if $\mathfrak{p}=(\mathfrak{p}_1, y_1)$, 
where  $\mathfrak{p}_1 \in \mathrm{Ass}(R/J^t)$.
 We thus have $\mathfrak{m}\notin \mathrm{Ass}(R/(J^t, y_1))$, and so   $\mathfrak{m}\notin \mathrm{Ass}(R/(I^t, y_1))$. Also, (\ref{12}) yields that  
$\mathrm{Ass}(R/I^t) \subseteq \mathrm{Ass}(R/(I^t:y_1)) \cup \mathrm{Ass}(R/(I^t, y_1))$. Accordingly, we obtain 
 $\mathfrak{m}\in \mathrm{Ass}(R/(I^t:y_1))$. 
Hence, the claim is true for the case in which $s=1$. Now, 
suppose  that the assertion has been shown  for a product of $s-1$ variables, and that $\mathfrak{m}\in \mathrm{Ass}(R/I^t)$. Set $M:=\prod_{i=1}^{s-1}y_i$. One can deduce from the induction that  
$\mathfrak{m}\in \mathrm{Ass}(R/(I^t:M))$. Moreover,    
\begin{equation}
 \mathrm{Ass}(R/(I^t:M))\subseteq \mathrm{Ass}(R/((I^t:M):y_s) )\cup \mathrm{Ass}(R/((I^t:M),y_s)).\label{13}
\end{equation} 
Let  $K:=I\setminus y_s$. It remains to check that  
$((I^t:M),y_s)=((K^t:M),y_s).$ Since $K\subseteq I$, one can easily conclude that  $((K^t:M),y_s) \subseteq ((I^t:M),y_s).$ 
To prove the reverse inclusion,  take a monomial $u$ in 
$((I^t:M),y_s)$. If $y_s\mid u$, then $u\in  ((K^t:M),y_s),$ 
and the argument is done. If $y_s\nmid u$, then 
$u\in  (I^t:M),$ and hence $uM\in I^t$. This yields that there exists a monomial $f\in \mathcal{G}(I^t)$ such that $f\mid uM$. Because $y_s\nmid u$ and $y_s \nmid M$, this implies that $y_s\nmid f$, and so $f\in K^t$. We therefore get $u\in (K^t:M)$, and thus  $((I^t:M),y_s) \subseteq ((K^t:M),y_s).$ 
It follows also from our assumption that $\mathfrak{m}\setminus y_s \notin \mathrm{Ass}(R/K^t)$. As  $\mathrm{Ass}(R/(K^t:M)) \subseteq \mathrm{Ass}(R/K^t)$, one can conclude that  $\mathfrak{m}\setminus y_s \notin \mathrm{Ass}(R/(K^t:M))$. Due to Theorem \ref{disjoint}, we obtain  
 $\mathfrak{p}\in \mathrm{Ass}(R/((K^t:M), y_s))$ 
if and only if $\mathfrak{p}=(\mathfrak{p}_1, y_s)$, 
where  $\mathfrak{p}_1 \in \mathrm{Ass}(R/(K^t:M))$.
This gives rise to 
$\mathfrak{m}\notin \mathrm{Ass}(R/((K^t:M), y_s))$, and so $\mathfrak{m}\notin\mathrm{Ass}(R/((I^t:M), y_s)).$ 
In light of $(\ref{13})$, one has 
$\mathfrak{m}\in \mathrm{Ass}(R/((I^t:M): y_s))=
\mathrm{Ass}(R/(I^t:\prod_{i=1}^sy_i))$.  
\end{proof}
 We are ready to express one of the main results of this section in the subsequent corollary.

 \begin{corollary}\label{Main.Corollary}
 Let $I\subset  R=K[x_1, \ldots, x_n]$ be a square-free  monomial ideal,  $\mathfrak{m}=(x_1, \ldots, x_n)$,   
 and  $\{u_1, \ldots, u_{\beta_1(I)}\}$  be a maximal independent set of minimal generators of $I$ such that  $\mathfrak{m}\setminus x_i \notin \mathrm{Ass}(R/(I\setminus x_i)^t)$ for all   $x_i\mid \prod_{i=1}^{\beta_1(I)}u_i$ and some positive integer $t$, where $I\setminus x_i$ denotes the  deletion of $I$ at $x_i$. 
 If $\mathfrak{m}\in \mathrm{Ass}(R/I^t)$, then $t\geq \beta_1(I)+1$. 
 \end{corollary}

\begin{proof}
To simplify the notation, set $\ell:=\beta_1(I)$. Let 
$\mathfrak{m}\in \mathrm{Ass}(R/I^t)$. 
Without loss of generality, one may assume that 
$\prod_{i=1}^{\ell}u_i=\prod_{i=1}^rx_i$.
In view of Theorem \ref{Th.Main}, we can conclude that 
$\mathfrak{m}\in \mathrm{Ass}(R/(I^t:\prod_{i=1}^rx_i)).$
  Since  $\prod_{i=1}^\ell u_i\in I^\ell$,  we get 
  $\prod_{i=1}^rx_i\in I^\ell$. If $t\leq \ell$, then 
 $(I^t:\prod_{i=1}^rx_i)=R$, which contradicts the fact that 
 $\mathfrak{m}\in \mathrm{Ass}(R/(I^t:\prod_{i=1}^rx_i)).$ This implies that $t\geq \ell +1$, as required.  
 \end{proof}


We apply  Corollary \ref{Main.Corollary} to study 
K$\mathrm{\ddot{o}}$nig ideals.

\begin{theorem} \label{Application}
Let $I$ be an unmixed K$\ddot{o}$nig ideal in the polynomial ring $R=K[x_1, \ldots, x_n]$ over a field $K$,   $\mathfrak{m}=(x_1, \ldots, x_n)$, and  $\{u_1, \ldots, u_{\beta_1(I)}\}$  be a maximal independent set of minimal generators of $I$ such that  $\mathfrak{m}\setminus x_i \notin \mathrm{Ass}(R/(I\setminus x_i)^t)$ for all $t$ and $x_i\mid \prod_{i=1}^{\beta_1(I)}u_i$.  Then the following statements hold:
\begin{itemize}
\item[(i)] $I$  is normally torsion-free. 
\item[(ii)]  $I$ is normal.
\item[(iii)]  $I$ has the strong persistence proeprty. 
\item[(iv)]  $I$ has the persistence property. 
\item[(v)]  $I$ has the symbolic strong persistence property. 
\end{itemize}

\end{theorem}

\begin{proof}
(i) For convenience of notation, put $\ell:=\beta_1(I)$. Suppose, on the contrary, that   $I$ is not normally torsion-free. Thus, there exists a positive integer $t$ such  that $I^t$ has embedded primes. We consider $t$ minimal with respect to this property.  
Let $\mathfrak{p}$ be an embedded prime of $I^t$. In view of Fact~\ref{fact4}, we get  
  $\mathfrak{p} \in \mathrm{Ass} (R/I^t)$ if and only if  
$\mathfrak{p}R_{\mathfrak{p}} \in \mathrm{Ass} (R_{\mathfrak{p}}/(IR_{\mathfrak{p}})^t)$.  This enables us to assume that $\mathfrak{p}=\mathfrak{m}$. 
Without loss of generality, one may assume that 
$\prod_{i=1}^{\ell}u_i=\prod_{i=1}^rx_i$. In  light of Theorem \ref{Th.Main}, one  has  $\mathfrak{m}\in \mathrm{Ass}(R/(I^t:\prod_{i=1}^rx_i))$, and so $\mathfrak{m}\in \mathrm{Ass}(R/(I^t:\prod_{i=1}^{\ell}u_i))$. 
On the other hand, since $\mathfrak{m}\setminus x_i \notin \mathrm{Ass}(R/(I\setminus x_i)^t)$ for each $i=1, \ldots, r$,  we deduce from Corollary \ref{Main.Corollary} that $t\geq \ell +1$. Because  $t$ is minimal,   $I^{t-\ell}$ has no embedded primes, and hence $I^{t-\ell}=I^{(t-\ell)}$. It follows now from \cite[Proposition 3.8]{HM} that $(I^t:\prod_{i=1}^{\ell}u_i)=I^{t-\ell}$. We therefore obtain $\mathfrak{m}\in \mathrm{Ass}(R/I^{t-\ell})$, which contradicts the fact that $I^{t-\ell}$ has no embedded primes. Consequently, $I$ is a normally torsion-free square-free monomial ideal. \par 
(ii) Based on    \cite[Theorem 1.4.6]{HH1}, every normally torsion-free square-free monomial ideal is normal. Now, the  assertion can be deduced from (i).  \par 
(iii) By   \cite[Theorem 6.2]{RNA}, every normal monomial ideal has the strong persistence property, and so the claim  follows readily from (ii). \par 
(iv)  By  \cite[Proposition 2.9]{N1}, the strong persistence property implies the persistence property. Hence, one  can    obtain the assertion from (iii). \par 
(v)  In view of  \cite[Theorem 11]{RT}, the strong persistence property implies the symbolic strong persistence property, and  thus the claim is  an immediate consequence of (iii).  
 \end{proof}

We are in a position to state the other main result of this section in the following theorem.  


\begin{theorem} \label{Th.Main2}
Let $I$ be a square-free  monomial ideal in a polynomial ring $R=K[x_1, \ldots, x_n]$ over a field $K$ and $\mathfrak{m}=(x_1, \ldots, x_n)$. If there exists a square-free monomial  $v \in I$ such that $v\in \mathfrak{p}\setminus \mathfrak{p}^2$ for any $\mathfrak{p}\in \mathrm{Min}(I)$, and $\mathfrak{m}\setminus x_i \notin \mathrm{Ass}(R/(I\setminus x_i)^s)$ for all $s$ and $x_i \in \mathrm{supp}(v)$, then the following statements hold:
\begin{itemize}
\item[(i)] $I$  is normally torsion-free. 
\item[(ii)]  $I$ is normal.
\item[(iii)]  $I$ has the strong persistence proeprty. 
\item[(iv)]  $I$ has the persistence property. 
\item[(v)]  $I$ has the symbolic strong persistence property. 
\end{itemize}
\end{theorem}
\begin{proof}
(i) Suppose, on the contrary, that $I$ is not normally torsion-free. 
Let $t$ be minimal such that $I^t$ has embedded prime ideals. 
 Since $I$ is a square-free monomial ideal, this implies that 
$\mathrm{Ass}(R/I)=\mathrm{Min}(I)$, and hence $t\geq 2$. 
Take an arbitrary $\mathfrak{p}\in \mathrm{Min}(I)$. In the sequel, we show that  $(I^t:v)=I^{t-1}$. Because   $v\in I$, one must have  $I^{t-1}\subseteq (I^t:v)$. To prove the reverse inclusion, select a monomial $f$ in $(I^t:v)$. Thus, $fv\in I^t$, and so there exist some monomials $g_1, \ldots, g_t \in \mathcal{G}(I)$ and some monomial $h$ in $R$ such that $fv=g_1 \cdots g_t h$. Due to $\mathfrak{p}\in \mathrm{Min}(I)$, this yields that $fv\in \mathfrak{p}^t$. Since $v\in \mathfrak{p}\setminus \mathfrak{p}^2$, one can conclude that $\mathfrak{p}$ contains exactly one variable that divides $v$. This gives that $f\in \mathfrak{p}^{t-1}$. Because  $\mathfrak{p}$ is arbitrary, one obtains that  $f\in \cap_{\mathfrak{p}\in \mathrm{Min}(I)}\mathfrak{p}^{t-1}$. Applying  \cite[Proposition 4.3.25]{V1},  $f\in I^{(t-1)}$.  Since $t$ is minimal such that $I^t$ has embedded prime ideals, this yields that $I^{t-1}$ has no embedded prime ideals, and so $I^{(t-1)}=I^{t-1}$. 
We thus gain $f\in I^{t-1}$. Consequently, $(I^t:v)\subseteq I^{t-1}$, and hence  $(I^t:v)= I^{t-1}$. Let $\mathfrak{q}$ be an embedded prime ideal of $I^t$. It follows from Fact~\ref{fact4} that $\mathfrak{q} \in \mathrm{Ass} (R/I^t)$ if and only if  
$\mathfrak{q}R_{\mathfrak{q}} \in \mathrm{Ass} (R_{\mathfrak{q}}/(IR_{\mathfrak{q}})^t)$. We therefore may assume that $\mathfrak{q}=\mathfrak{m}$. Hence, $\mathfrak{m}\in \mathrm{Ass}(R/I^t)$. In view of Theorem \ref{Th.Main}, one derives that 
$\mathfrak{m}\in \mathrm{Ass}(R/(I^t:v))$. We thus get 
$\mathfrak{m}\in \mathrm{Ass}(R/I^{t-1})$. This contradicts our assumption. Accordingly, the square-free monomial ideal $I$ is normally torsion-free. \par 
(ii)-(v) The claims can be concluded by similar  discussions in the proof of Theorem \ref{Application}.
\end{proof}

To illustrate  Theorem \ref{Th.Main2}, we provide the following example. First we recall  the definition of $t$-spread monomial ideals.

\begin{definition} (see \cite{EHQ}) 
A monomial $x_{i_1}x_{i_2}\cdots x_{i_d}$ with $i_1\leq i_2 \cdots \leq i_d$ is called {\it $t$-spread}  if $i_j-i_{j-1}\geq t$ for $2\leq j \leq d$. Also, a monomial ideal in $R=K[x_1, \ldots, x_n]$ is called a {\it $t$-spread monomial ideal}, if it is generated by $t$-spread monomial ideals.  
\end{definition}

\begin{example}
Let $I$ be a square-free monomial ideal in the polynomial ring $ R=K[x_1, \ldots, x_7]$ over a field $K$ with 
$$I = (x_1x_3x_6, x_1x_3x_7, x_1x_4x_6, x_1x_4x_7, 
  x_1x_5x_7, x_2x_4x_7, x_2x_5x_7).$$
Direct computations show that $I$ is a $2$-spread monomial ideal.
 By using Macaulay2 \cite{GS},  we obtain 
$$\mathrm{Ass}(R/I)=\{(x_1,x_2), (x_1,x_7), (x_6, x_7), (x_1, x_4, x_5), (x_3, x_4, x_7), (x_3, x_4, x_5)\}.$$
 Also, note that $I$ is not unmixed. Set $v:=x_1x_3x_6$.  One can easily check that $v\in \mathfrak{p}\setminus \mathfrak{p}^2$ for each $\mathfrak{p}\in \mathrm{Min}(I)$. 
 To complete our argument, one needs to demonstrate that 
  $\mathfrak{m}\setminus x_i \notin \mathrm{Ass}(R/(I\setminus x_i)^s)$ for all $s$ and $i\in \{1,3,6\}$, where 
 $\mathrm{m}=(x_1, \ldots, x_7)$. 
We first note that  $I\setminus x_1=(x_2x_4x_7, x_2x_5x_7)$. Since $(x_2x_4x_7, x_2x_5x_7)=x_2x_7(x_4,x_5)$, it follows from Fact~\ref{fact5} that $(x_2x_4x_7, x_2x_5x_7)$ is normally torsion-free, and so $\mathfrak{m}\setminus x_1 \notin \mathrm{Ass}(R/(I\setminus x_1)^s)$ for all $s$. Next, observe that 
$$F:= I\setminus x_3=(x_1x_4x_6, x_1x_4x_7, x_1x_5x_7, x_2x_4x_7, x_2x_5x_7).$$ 
Once again, by using Macaulay2 \cite{GS}, one has 
$$\mathrm{Ass}(R/(I\setminus x_3))=\{(x_1,x_2),  (x_4, x_5), (x_4, x_7), (x_6, x_7), (x_1, x_7)\}.$$
Let $G=(V(G), E(G))$ be the graph with the vertex set 
$V(G)=\{1,2,4,5,6,7\}$ and the edge set 
$E(G)=\{\{1,2\}, \{4,5\}, \{4,7\}, \{6,7\}, \{1,7\}\}$. It is routine to check that  $G$ is a tree, and  $F$ is its cover ideal. 
By virtue of  \cite[Corollary 2.6]{GRV},  $F$ is normally torsion-free; thus, $\mathfrak{m}\setminus x_3 \notin \mathrm{Ass}(R/(I\setminus x_3)^s)$ for all $s$. Finally, 
notice that 
$$ I\setminus x_6=(x_1x_3x_7, x_1x_4x_7, x_1x_5x_7, x_2x_4x_7, x_2x_5x_7).$$
 Put $ L:=(x_1x_3, x_1x_4, x_1x_5, x_2x_4, x_2x_5).$
 Assume that $H=(V(H), E(H))$ is  the graph with the vertex set 
$V(H)=\{1,2,3,4,5\}$ and the edge set 
$$E(H)=\{\{1,3\}, \{1,4\}, \{1,5\}, \{2,4\}, \{2,5\}\}.$$  $H$ has no odd cycle subgraph, so     $H$ is bipartite. In addition,  $L$ is the edge ideal of $H$, and  so \cite[Theorem 5.9]{SVV} yields that $L$ is normally torsion-free. 
We therefore get $\mathfrak{m}\setminus x_6 \notin \mathrm{Ass}(R/(I\setminus x_6)^s)$ for all $s$.
Due to Fact~\ref{fact5}, one has $I\setminus x_6=x_7L$  is normally torsion-free. It follows now from Theorem \ref{Th.Main2} that $I$ is normally torsion-free, as desired. 
\end{example}

We terminate this section by giving an application of Theorem 
\ref{Th.Main2}. In fact, we re-prove the fact that every square-free transversal polymatroidal ideal is normally torsion-free, see \cite[Corollary 3.6]{HRV}. 
To accomplish this, one has to recall the following definition.
\begin{definition} (see  \cite{HRV}) 
Let $F$ be a non-empty subset of $[n]=\{1, \ldots, n\}$. We denote by $\mathfrak{p}_F$ the monomial prime ideal $({x_j : j\in  F}).$ 
A  {\it transversal polymatroidal ideal} is an ideal $I$ of the
form $I=\mathfrak{p}_{F_1} \cdots \mathfrak{p}_{F_r},$
where $F_1,\ldots, F_r$ is a collection of non-empty subsets of $[n]$ with $r\geq 1$.   
\end{definition}

 \begin{theorem}\label{Application2}
Every square-free transversal polymatroidal ideal is normally torsion-free. 
\end{theorem}

\begin{proof}
Let $I\subset  R=K[x_1, \ldots, x_n]$ be a square-free transversal polymatroidal ideal.  Thus, we can write 
$I=\mathfrak{p}_{F_1} \cdots \mathfrak{p}_{F_r},$
where $F_1,\ldots, F_r$ is a collection of non-empty subsets of $[n]$ with $r\geq 1$ and $F_i\cap F_j=\emptyset$ for all $1\leq i\neq j \leq r$. Without loss of generality, one may assume that $\cup_{u\in \mathcal{G}(I)}\mathrm{supp}(u)=\{x_1, \ldots, x_n\}$. 
We proceed by induction on $n$. If $n=1$, then $I=(x_1)$ which is normally torsion-free.  Now, suppose that $n>1$, and that the claim has been proven for $n-1$.  Without loss of generality, we may assume that $x_i\in \mathfrak{p}_{F_i}$ for each $i=1, \ldots, r$. Set  $v:=x_1 \cdots x_r$. Note that 
$\mathrm{Min}(I)=\{\mathfrak{p}_{F_1}, \ldots, \mathfrak{p}_{F_r}\}$. It is routine to check that  $v\in \mathfrak{p}_{F_i}\setminus \mathfrak{p}^2_{F_i}$ for each $i=1, \ldots, r$. Fix $1\leq i \leq r$. Since 
$$I\setminus x_i=\mathfrak{p}_{F_1} \cdots \mathfrak{p}_{F_{i-1}} 
(\mathfrak{p}_{F_i}\setminus \{x_i\})   \mathfrak{p}_{F_{i+1}}\cdots  \mathfrak{p}_{F_r},$$ it follows from the induction  hypothesis that $I\setminus x_i$ is normally torsion-free. This leads to $\mathfrak{m}\setminus x_i \notin \mathrm{Ass}(R/(I\setminus x_i)^s)$ for all $s$ and $x_i \in \mathrm{supp}(v)=\{x_1, \ldots, x_r\}$, where $\mathfrak{m}=(x_1, \ldots, x_n)$.  By   Theorem \ref{Th.Main2},   $I$ is normally torsion-free. 
\end{proof}


 \section{Some results on the corner-elements of monomial ideals}

Let $I$ be a monomial ideal in a polynomial ring $R=K[x_1, \ldots, x_n]$ over a field $K$ and $\mathfrak{p}\in \mathrm{Ass}(R/I)$. It is well-known by \cite[Corollary 1.3.10] {HH1} that $\mathfrak{p}=(I:h)$ for some monomial $h$ in $R$. Moreover, it has been proven that if 
 $I\subset R=K[x_1, \ldots, x_n]$ is  a square-free monomial ideal  and $\mathfrak{p}\in \mathrm{Ass}(R/I)$, then there exists a square-free  monomial $h$ in  $R$ with $\mathrm{supp}(h)\subseteq \cup_{u\in \mathcal{G}(I)}\mathrm{supp}(u)$ such that $\mathfrak{p}=(I:_R h)$, see \cite[Theorem 3.1]{KHN2}.  In this section, we deduce 
 which variables must divide   $h$. To  accomplish this, we begin  with the following proposition.

\begin{proposition} \label{Pro.Corner.1}
Let $I$ be a monomial ideal in a polynomial ring $R=K[x_1, \ldots, x_n]$ over a field $K$. Let $\mathfrak{p},   \mathfrak{q} \in \mathrm{Ass}(R/I^t)$ with $\mathfrak{p}\neq \mathfrak{q}$ such that $\mathfrak{p}=(I^t:f)$ and $\mathfrak{q}=(I^t:g)$ for some positive integer $t$ and some monomials $f$ and $g$ in $R$. Then $f\nmid g$ and $g\nmid f$.  
\end{proposition}
\begin{proof}
We need only show that  $f\nmid g$, since the other claim can be proven  by a similar argument. Suppose, on the contrary, that $f\mid g$. That is, $g=fv$ for some monomial $v$ in $R$.   This gives that $\mathfrak{q}=(I^t:fv)$. Since $(I^t:fv)=((I^t:f):v)$, one gains that   $\mathfrak{q}=(\mathfrak{p}:v)$. It is easy to see that either 
$(\mathfrak{p}:v)=R$ or  $(\mathfrak{p}:v)=\mathfrak{p}$. 
We thus have either $\mathfrak{q}=R$ or $\mathfrak{q}=\mathfrak{p}$. This leads to a contradiction. Therefore, we get  $f\nmid g$, as claimed. 
\end{proof}


 To establish Proposition \ref{Pro.Corner.2} and Corollary \ref{Main.Corollary2},  one has to  apply the notion of  corner-elements which were first introduced in \cite{HRS}. We recall it in the following definition.

 \begin{definition} 
   Let  $R = K[x_1, \ldots, x_n]$ be a polynomial ring, and $I$ be an ideal of $R$. A monomial $z$ in $R$ is called  an {\it I-corner-element} if $z\notin I$ and $x_iz\in I$ for each $i=1, \ldots, n$.   
 \end{definition}

\begin{proposition} \label{Pro.Corner.2}
 Let $I\subset R=K[x_1, \ldots, x_n]$ be a  monomial ideal and  $\mathfrak{m}=(x_1, \ldots, x_n)$. Let $f$ and $g$ be two $I^t$-corner-elements for some positive integer $t$. Then 
$f\nmid g$ and  $g\nmid f$. 
\end{proposition}
\begin{proof}
Since $f$ and $g$ are two  $I^t$-corner-elements for some positive integer $t$, this implies that $f\notin I^t$ (respectively, $g\notin I^t$) and $x_if\in I^t$ (respectively, $x_ig\in I^t$) for each $i=1, \ldots, n$. We thus have $x_i\in (I^t:f)$ (respectively, $x_i\in (I^t:g)$) for each $i=1, \ldots, n$. 
Accordingly, $\mathfrak{m}\subseteq (I^t:f)$ (respectively, $\mathfrak{m}\subseteq (I^t:g)$). Because 
$f\notin I^t$ (respectively, $g\notin I^t$), this yields that $\mathfrak{m}=(I^t:f)$ (respectively, $\mathfrak{m}=(I^t:g)$). On the contrary, assume that $f\mid g$. This gives that $g=fh$ for some monomial $h$ in $R$. We thus get $\mathfrak{m}=((I^t:f):v)$, and so $\mathfrak{m}=R$, a contradiction. A similar discussion shows that $g\nmid f$, and the proof is complete.
\end{proof}

The following lemma will be useful in the proof of  Corollary \ref{Main.Corollary2}. 

\begin{lemma}\label{Lemma.Corner}
Let $I$ be a monomial ideal in a polynomial ring $R=K[x_1, \ldots, x_n]$ over a field $K$.  Let  
$\mathfrak{p}=(I^t:h)$ for some positive integer $t$ and some monomial $h$ in $R$ such that  $x_i\nmid h$ for some $1\leq i \leq n$. Then $\mathfrak{p}\setminus x_i=((I\setminus x_i)^t:h)$, and so  $\mathfrak{p}\setminus x_i \in \mathrm{Ass}(I\setminus x_i)^t$.
\end{lemma}
\begin{proof}
 We show that   $\mathfrak{p}\setminus x_i =((I\setminus x_i)^t:h)$. To do this, we first prove that 
$\mathfrak{p}\setminus x_i \subseteq ((I\setminus x_i)^t:h)$.
 To see this, take $x_j \in \mathfrak{p}\setminus x_i$. Since 
 $\mathfrak{p}\setminus x_i \subseteq \mathfrak{p}$, this implies that $x_j \in \mathfrak{p}$, and so $x_jh\in I^t$. Because $x_j\neq x_i$ and $x_i\nmid h$, one can conclude that $x_i\nmid x_jh$. Hence, $x_jh\in I^t\setminus x_i$. 
 Since  $I^t\setminus x_i=(I\setminus x_i)^t$, we get $x_jh \in (I\setminus x_i)^t$, and thus $x_j \in ((I\setminus x_i)^t:h)$.  We therefore have 
 $\mathfrak{p}\setminus x_i \subseteq ((I\setminus x_i)^t:h)$. Conversely, let $v\notin \mathfrak{p}\setminus x_i$. It suffices to show  that $v\notin ((I\setminus x_i)^t:h)$. Suppose, on the contrary, that  
 $v\in ((I\setminus x_i)^t:h)$. Then 
  $vh\in (I\setminus x_i)^t$. Since $v\in \mathfrak{p}$ and 
$v\notin \mathfrak{p}\setminus x_i$, one must have $v=x^\ell_ia$ for some positive integer $\ell$ and some monomial $a\notin \mathfrak{p}$. Hence, $x^\ell_iah\in (I\setminus x_i)^t$, and so $ah\in ((I\setminus x_i)^t:x^\ell_i)$.  It follows readily from  Lemma \ref{Kaplansky}  that $ah\in (I\setminus x_i)^t$. Since  $I\setminus x_i \subseteq I$, we obtain that $ah\in I^t$, and so $a\in (I^t:h)$. Accordingly,  one has $a\in \mathfrak{p}$, 
which is a contradiction. We  thus have $v\notin ((I\setminus x_i)^t:h)$, and so 
$((I\setminus x_i)^t:h) \subseteq \mathfrak{p}
\setminus x_i$. Consequently,  
$\mathfrak{p}\setminus x_i=((I\setminus x_i)^t:h)$, and hence $\mathfrak{p}\setminus x_i \in \mathrm{Ass}(I\setminus x_i)^t$.
\end{proof}

We are ready to express one of the main results of this section in the subsequent corollary.

\begin{corollary}
Let $I$ be a monomial ideal in a polynomial ring $R=K[x_1, \ldots, x_n]$ over a field $K$.  If  
$\mathfrak{p}=(I^t:h)$ for some positive integer $t$ and some monomial $h$ in $R$ such that  $\mathfrak{p}\setminus x_i \notin \mathrm{Ass}(I\setminus x_i)^t$ for some $1\leq i \leq n$, then $x_i\mid h$.
\end{corollary}

We are in a position to state the other main result of this section on the corner-elements of monomial ideals in the following corollary.

\begin{corollary}\label{Main.Corollary2}
Let $I$ be a monomial ideal in a polynomial ring $R=K[x_1, \ldots, x_n]$ over a field $K$ and $z$ be an $I^t$-corner-element  for some positive integer $t$  such that  $\mathfrak{m}\setminus x_i \notin \mathrm{Ass}(I\setminus x_i)^t$ for some $1\leq i \leq n$. Then $x_i\mid z$.
\end{corollary}

The following question posed by Rajaee, Nasernejad, and Al-Ayyoub in \cite{RNA} asks if  Lemma \ref{Lemma.Corner} can be strengthened further.
\begin{question}\label{Question1}
\cite[Question 4.13]{RNA} 
Suppose that  $I$ is   a square-free  monomial ideal in  $R=K[x_1, \ldots, x_n]$, $\mathcal{G}(I)=\{u_1, \ldots, u_m\}$, $\bigcup_{i=1}^m\mathrm{supp}(u_i)=\{x_1, \ldots, x_n\}$, and  
  $\mathfrak{m}=(x_1, \ldots, x_n)$ is the graded maximal ideal of $R$. If there exists a positive integer  $1\leq j \leq n$ such that  
  $\mathfrak{m}\setminus x_j\in \mathrm{Ass}(R/(I\setminus x_j)^k)$  for some positive integer $k$, then can we deduce that  $\mathfrak{m}\in \mathrm{Ass}_R(R/I^k)$?   
\end{question}

We provide a counterexample. To do this, let $G=(V(G), E(G))$ be the graph with the vertex set $V(G)=\{1, 2, 3, 4, 5, 6\}$, and the edge set 
$$E(G)=\{\{1,2\}, \{2,3\}, \{3,4\}, \{4,5\}, \{5,1\},\{1,6\},\{5,6\}, \{4,6\}\}.$$ 
Assume that $I$ is the cover ideal of $G$, that is,
$$I=(x_1,x_2) \cap  (x_2,x_3) \cap  (x_3,x_4) \cap 
 (x_4,x_5) \cap  (x_5,x_1) \cap (x_1,x_6) \cap  (x_5,x_6) 
 \cap  (x_4,x_6).$$
  By using Macaulay2 \cite{GS},  one can derive that 
 $(x_1, x_2, x_3, x_4, x_5)\in \mathrm{Ass}(R/(I\setminus x_6)^2)$, while   $(x_1, x_2, x_3, x_4, x_5, x_6)\notin \mathrm{Ass}(R/I^2)$.
\\

The proposition below says that if $\mathfrak{p}=(I^s:h)\in \mathrm{Ass}(R/I^s)$ for some positive integer $s$, then we can always find an upper bound for $\mathrm{deg}_{x_i}h$ for each $i$.

\begin{proposition}\label{Pro.Corner.3}
Let $I$ be a square-free monomial ideal in a polynomial ring $R=K[x_1, \ldots, x_n]$ over a field $K$.  Let   
$\mathfrak{p}=(I^s:h)\in \mathrm{Ass}(R/I^s)$ for some positive integer $s$ and some monomial $h$ in $R$. Then $\mathrm{deg}_{x_i}h\leq s-1$ for each $i=1, \ldots, n$. 
\end{proposition}
\begin{proof}
To  verify  the claim, it is enough for us to  show that 
 $h|(\prod_{i=1}^nx_i)^{s-1}$. On the contrary, assume that $h\nmid(\prod_{i=1}^nx_i)^{s-1}$. Thus,  there exists some $1\leq i \leq n$ such that $\mathrm{deg}_{x_i}h\geq s$. This implies that  the exponent of $x_i$ in $hx_i$ is at least $s+1$. By $\mathfrak{p}=(I^s:h)$, one has   $hx_i \in I^s$, and so   there exist monomials $u_1, \ldots, u_s\in \mathcal{G}(I)$  such that 
 $hx_i=u_1\cdots u_sf$ for some monomial $f$ in $R$. 
 Since  $I$ is  a square-free monomial ideal, then for each $j=1,\ldots, s$, $u_j$ is a square-free monomial, one has the exponent of $x_i$ in each $u_j$ is at most one. Accordingly, we get $x_i|f$, and hence  $h\in I^s$, which is a contradiction.  We therefore have    $h|(\prod_{i=1}^nx_i)^{s-1}$, as desired. 
\end{proof}

We close this section with the subsequent proposition which argues on the minimal prime ideals associated to  square-free monomial ideals. 

\begin{proposition}\label{Pro.Corner.4}
Let $I$ be a square-free monomial ideal in a polynomial ring $R=K[x_1, \ldots, x_n]$ over a field $K$. If  $\mathfrak{p}\in \mathrm{Min}(I)$, then $\mathfrak{p}=(I:h)$, where $h$ is the product of the variables that are not in $\mathfrak{p}$. 
\end{proposition}
\begin{proof}
Let $\mathfrak{p}\in \mathrm{Min}(I)$.  We first  demonstrate 
that  $(I:h) \subseteq \mathfrak{p}.$ To see this, pick a monomial $u$ in $(I:h)$. Hence, $uh\in I$, and so $uh \in \mathfrak{p}$. We thus have $x_j\mid uh$ for some $x_j \in 
\mathfrak{p}.$  Since $\mathrm{gcd}(x_j, h)=1$, one can conclude that $x_j\mid u$, and so $u\in \mathfrak{p}$. Therefore,  $(I:h) \subseteq \mathfrak{p}.$ To complete our discussion, 
we show that $\mathfrak{p}\subseteq (I:h)$. To achieve this, take an arbitrary element $x_j\in  \mathfrak{p}$.   Because 
$\mathfrak{p}\in \mathrm{Min}(I)$, this yields that $I\nsubseteq \mathfrak{p}\setminus \{x_j\}$, and so there exists some $u\in \mathcal{G}(I)$ such that $u\notin \mathfrak{p}\setminus \{x_j\}$. As $I$ is  square-free, this forces  $u\mid x_jh$, and hence $x_jh\in I$. Accordingly, one must have $x_j \in (I:h)$. We therefore get $\mathfrak{p}\subseteq (I:h)$, as required. 
\end{proof}
  \noindent{\bf Acknowledgments.}
The authors are deeply grateful to the  referee for careful reading of the manuscript and    valuable suggestions which led to significant  improvements in the quality of this paper.  
Mirsadegh  Sayedsadeghi  was supported by the grant of Payame Noor University of Iran.
Moreover, Mehrdad Nasernejad  was in part supported by a grant from IPM (No. 991300120).


\end{document}